\documentclass[letterpaper,10pt]{amsart}

\usepackage[utf8]{inputenc} 
\usepackage[T1]{fontenc}
\usepackage{url}
\usepackage[margin=3cm]{geometry}
\usepackage{parskip}
\usepackage[figuresright]{rotating} 

\usepackage{amssymb} 
\usepackage{amsmath}
\usepackage{amsthm} 
\usepackage{mathabx} 
\usepackage{mathtools}
\usepackage{hyperref}

\allowdisplaybreaks 

\usepackage{enumerate}
\usepackage[shortlabels]{enumitem}

\newcommand{\Z}{\mathbb{Z}}
\newcommand{\inv}{^{-1}}
\newcommand{\del}{\partial}
\newcommand{\sys}{\operatorname{sys}}

\newtheorem{theor}{Theorem} 
\newtheorem{thm}{Theorem}[section]
\newtheorem{prop}[thm]{Proposition}

\newtheorem{lemma}[thm]{Lemma}

\newtheorem*{thm*}{Theorem}
\newtheorem*{prop*}{Proposition}
\newtheorem*{cor*}{Corollary}

\theoremstyle{definition}

\theoremstyle{remark}

\title{On the number of shared Dehn surgeries between two knots}
\author{Patricia Sorya}
\address{University of Ottawa, Department of Mathematics and Statistics}
\email{psorya@uottawa.ca}

\thanks{This work was supported by the FRQ doctoral research grant 305903 and by the NSERC postdoctoral fellowship 598649}

\date{\today}

\begin{document}

\begin{abstract}
A folklore theorem states that for any pair of distinct knots in $S^3$, performing $p/q$-Dehn surgery on each knot yields orientation-preservingly homeomorphic manifolds for at most finitely many slopes $p/q$. In this paper, we provide a proof based on the JSJ decomposition of knot exteriors. In particular, for any given pair of distinct knots, it provides an effective bound on the maximal number of shared surgeries between the knots.
\end{abstract}
\maketitle

The $p/q$-\emph{Dehn surgery} along a knot $K$ in $S^3$, denoted by $S^3_K(p/q)$, is obtained by gluing a solid torus $S^1 \times D^2$ to the exterior $S^3_K$ of $K$ in $S^3$. The parameter $p/q$ is called a \emph{slope} and specifies the curve $p\mu+q\lambda \subset \del S^3_K$ to which a meridional disc of $S^1 \times D^2$ is identified, where $\mu$ and $\lambda$ are the meridian and longitude of $K$ respectively.

A folklore theorem \cite{bkm, kp}, states that given two distinct knots $K$ and $K'$, the manifolds $S^3_K(p/q)$ and $S^3_{K'}(p/q)$ are orientation-preservingly homeomorphic for at most finitely many slopes $p/q$. To the best of our knowledge, this statement is unattributed and no complete proof is known. In this paper, we provide a proof based on the JSJ decomposition of knot exteriors. By combining ideas from previous work \cite{sorya, sw}, we obtain an effective upper bound on the number of slopes for which an homeomorphism $S^3_K(p/q) \cong S^3_{K'}(p/q)$ can occur.

\begin{theor}
	\label{thm:main} 
	Let $K$ and $K'$ be knots in $S^3$. 
	There are at most $D(K,K')$ slopes $p/q$ such that $S^3_K(p/q) \cong S^3_{K'}(p/q)$,
	where $D(K,K')$ is specified in Table \ref{table:D}.
\end{theor}

When at least one of $K$ and $K'$ is a prime knot, the construction of $D(K,K')$ relies on the topology of JSJ pieces appearing in the exterior of the knots and their fillings. In some cases, we find a precise set of slopes that may yield shared surgeries. In other cases, we find numbers $C$ and $B$ such that $S^3_K(p/q) \cong S^3_{K'}(p/q)$ only if $|q|\leq C$ and $|p| \leq B$. We then take $D(K,K')$ to be the number
	\[n(B,C) = \# \big\{ (p,q) \in \Z_{\geq 0} \times \Z\setminus\{0\} \;\big|\; \gcd(p,q)=1,  p\leq B,  |q| \leq C \big\}\]

Note that such a constant $C$ can be obtained from \cite[Theorem 1.6]{sw} or \cite[Definition A.2]{bkm}. In the present paper, we rely on the data of exceptional Seifert fibres and hyperbolic geodesics created by the surgery operation as in \cite[Theorem 1.6]{sw}, but we use specific geometric data instead of universal bounds. This strategy provides a slight refinement for $C$ and also has the advantage of being applicable to the construction of the constant $B$.

As an example, \cite[Theorem 0.1]{kp} exhibits hyperbolic knots $K_2$ and $K_2'$ that share $p/1$-Dehn surgeries for $p \in \{3,4,5,6\}$. Our bound gives $D(K_2,K_2') = n(7, 20)=147$. Its construction produces a finite list of slopes $p/q$ which can be computationally tested for yielding orientation-preservingly homeomorphic $p/q$-Dehn surgeries. It turns out that the only slopes with this property are precisely the 4 slopes from \cite[Theorem 0.1]{kp}\footnote{Communicated by Marc Kegel.}.

In the case when both $K$ and $K'$ are composite knots, for which only integer slopes may yield homeomorphic surgeries \cite[Theorem 2]{sorya}, no new exceptional Seifert fibre nor hyperbolic geodesic is introduced by the surgery operation. Thus, a different strategy is adopted: we analyze the gluing map prescribed by the surgery slope to obtain a bound $D(K,K')$ that depends on the number of summands of hyperbolic type.

\subsection{Organization}
We first gather some useful lemmas in the Section \ref{sec:pre}. We then study each possible combination of torus knot, hyperbolic type knot, cable knot and composite knot for $K$ and $K'$. The section corresponding to each case is indicated in Table \ref{table:org}.

\begin{table}[h!]
	\begin{tabular}{l|llll}
		$K$ \textbackslash{} $K'$ & Torus & Hyperbolic type & Cable & Composing \\ \hline
		Torus                     &  \ref{tt}   & \ref{th} & \ref{tc} &  \ref{tp}   \\
		Hyperbolic type               &       &  \ref{hh}   & \ref{hc}  &  \ref{hp}    \\
		Cable                     &       &            & \ref{cc}  &  \ref{cp}    \\
		Composing                 &       &            &       &     \ref{pp}
	\end{tabular}\caption{Section for each case}\label{table:org}
\end{table}

\subsection{Acknowledgements}
I thank Giacomo Bascapè, Steven Boyer, Duncan McCoy and Laura Wakelin for interesting discussions, Marc Kegel for pointing out the possibility of computationally testing candidate slopes, and Gary Dunkerley for asking me if I knew a proof of this folklore theorem.

\section{Preliminaries}\label{sec:pre}
We refer the reader to \cite[Section 2]{sw} for a summary of the structure of satellite knots, the JSJ decomposition of knot exteriors, hyperbolic geometry in knot exteriors, as well as for the terminology and notational conventions used in this paper.

We recall a useful geometric lemma for hyperbolic link exteriors.

\begin{lemma}\label{lemma:length}
	Let $\sigma=a/b$ be a slope on a boundary component of the exterior of a link in $S^3$, expressed in the basis given by the meridian $\mu$ and the longitude $\lambda$ of the corresponding link component. We have
	\begin{enumerate}[(i)]
		 \item $\hat{l}(\sigma)\hat{l}(\lambda) \geq \Delta(\sigma, \mu) = |a|$,
		 \item $\hat{l}(\sigma)\hat{l}(\mu) \geq \Delta(\sigma, \mu) = |b|$
\end{enumerate}		 
	where $\hat{l}(\gamma)$ denotes the normalized length of a slope $\gamma$.
\end{lemma}

Combining this with the 6-theorem \cite{agol, lackenby_6}, \cite[Theorem 1.2]{lm} and \cite[Theorem 2.8]{lackenby}, we obtain bounds ensuring that a filling is not exceptional. Recall that each JSJ piece $Y$ of a knot exterior is associated to a unique link $L_Y$ up to isotopy \cite[Proposition 2.4]{budney}. Let $\mu_Y$ and $\lambda_Y$ be the meridian and longitude of the link component corresponding the outermost boundary component $T$ of $Y$.

Define

\begin{align*}
    a(Y) &= \frac{6l(\lambda_Y)}{A(T)}, \\
    e(Y) &= \begin{cases}
        \min\{8, \frac{6l(\mu_Y)}{A(T)}\} \;, &|\del Y| = 1;\\
        \min\{2, \frac{6l(\mu_Y)}{A(T)}\} \;, &|\del Y| > 1.
    \end{cases}
\end{align*}

\begin{lemma}\label{lemma:hyperbolicfilling}
	Let $Y$ be a hyperbolic JSJ piece of a knot exterior and let $Y(p/q)$ denote its $p/q$-Dehn filling along its outermost boundary component $T$. If $|p|>a(Y)$ or $|q|>e(Y)$, then $Y(p/q)$ is hyperbolic.	
\end{lemma}
\begin{proof}
	Suppose $|p|>a(Y)$ and let $\sigma=p/q$. By Lemma \ref{lemma:length}(i),
	\[\hat{l}(\sigma) \geq \frac{|p|}{\hat{l}(\lambda)} \quad\Rightarrow\quad \frac{l(\sigma)}{\sqrt{A(T)}} \geq |p| \frac{\sqrt{A(T)}}{l(\lambda)} \quad\Rightarrow\quad l(\sigma) \geq |p| \frac{A(T)}{l(\lambda_Y)} >  \frac{6 l(\lambda) }{A(T)} \frac{A(T)}{l(\lambda)} = 6.\]
	The 6-theorem then implies that $Y(p/q)$ is hyperbolic.
	
	The case $|q|>e(Y)$ is Lemma \ref{lemma:length}(ii) applied analogously, combined with \cite[Theorem 1.2]{lm} and \cite[Theorem 2.8]{lackenby}.
\end{proof}

The following lemma describes the manifold $S^3_{K}(p/1)$ when $K$ is a composite knot.

\begin{lemma}\label{lemma:compositegluing}
	Let $K = K_1 \# K_2$ be a composite knot. The manifold $S^3_K(p/1)$, is homeomorphic to $S^3_{K_1} \cup_\phi S^3_{K_2}$ where the gluing map $\phi$ is given by
	\[\begin{array}{cccc}
		{\phi:} & {\pi_1(\partial S^3_{K_1})} & \longrightarrow & {\pi_1(\partial S^3_{K_2})} \\
		& {\mu_1} & \longmapsto & {\mu_2} \\
		& {\lambda_1} & \longmapsto & {-p\mu_2-\lambda_2}
	\end{array}\]
\end{lemma}
\begin{proof}
	Recall that $S^3_K$ can be described $S^3_{K_1} \cup (F\times S^1) \cup S^3_{K_2}$, where $F$ is a planar surface with three boundary components. Let $T_0, T_1, T_2$ be the boundary components of $F \times S^1$ such that $T_i = \del S^3_{K_i}$ in $S^3_K$ and $T_0 = \del S^3_K$.

	For $i=1,2$, let $\mu_i$ be the meridian of $K_i$, which is identified with the slope $h$ of a regular fibre of $F\times S^1$ along $T_i$, and let $\lambda_i$ be the longitude of $K_i$ on $T_i$.
	
	We have that $\{h, \lambda_i\}$ generates $\pi_1(T_i; \Z)$ for $i=1,2$, and $\{h, \lambda_1+\lambda_2\}$ generates $\pi_1(T_0; \Z)$. Further, the images induced by inclusion of $h, \lambda_1, \lambda_2$ into $F \times S^1$ generate $\pi_1(F \times S^1; \Z)$.
	
	The $p/q$-surgery along $K$ is performed by filling $F\times S^1$ along $T_0$, adding the relation $ph + q(\lambda_1+\lambda_2)$ to $\pi_1((F\times S^1)(T_0; p/q); \Z)$. In $S^3_K(p/q) \cong S^3_{K_1} \cup (F\times S^1)(T_0; p/q) \cup S^3_{K_2}$, we thus have that $\mu_1$ and $\mu_2$ are homotopic, and $q\lambda_1$ is homotopic to $-p\mu_2 - q\lambda_2$ (and vice versa).
	
	If $q=1$, then $(F\times S^1)(T_0; p/1)$ is homeomorphic to $T^2 \times I$ and therefore $S^3_K(p/1) \cong S^3_{K_1} \cup_{\phi} S^3_{K_2}$ with $\phi$ as in the statement.
\end{proof}
\section{Torus knot}
Let $K$ be an $(a,b)$-torus knot.

\subsection{Torus knot}\label{tt} Let $K' \neq K$ be a torus knot. By \cite[Lemma 4.2]{mccoy}, there are only 2 slopes that may yield orientation-preserving homeomorphic surgeries.

\subsection{Hyperbolic type}\label{th} Let $K'$ be of hyperbolic type, with outermost piece $Y'$.
For any $p/q$, the manifold $S^3_K(p/q)$ is never hyperbolic. So any slope $p/q$ such that $S^3_K(p/q) \cong S^3_{K'}(p/q)$ must be an exceptional slope for $Y'$.
By Lemma \ref{lemma:hyperbolicfilling}, $Y'$ has at most
\[n\left(a(Y'), \; e(Y')\right)\]
exceptional fillings, and therefore there are at most that many shared surgeries between $K$ and $K'$.

\subsection{Cable knot}\label{tc} Let $K'$ be an $(r',s')$-cable knot $C_{r',s'}(J')$.
By \cite[Proposition 1.5]{mccoy}, there is an orientation-preserving homeomorphism $S^3_{K}(p/q) \cong S^3_{K'}(p/q)$  for at most one non-integral slope $p/q$.

Furthermore, each of $K$ and $K'$ admit a unique slope yielding a reducible surgery -- $ab/1$ and $rs/1$ respectively  -- corresponding to the slope of a regular fibre along the boundary of their exterior. If $p/q = ab/1 = rs/1$ and $S^3_{K}(p/q) \cong S^3_{K'}(p/q)$, then by \cite[Proposition 4]{moser} and \cite[Corollary 7.3]{gordon_sat}  we have $L_{a,b} \# L_{b,a} \cong S^3_{J'}(r/s) \# L_{s,r}$. Therefore, without loss of generality, $L_{a,b} \cong S^3_{J'}(r/s)$ and comparing homology gives $|a|=|r|$. Since $ab=rs$, we have $r/s=a/b$ and therefore $S^3_{J'}(a/b) \cong L_{a,b}$ which implies that $J'$ is the unknot \cite[Theorem 1.1]{kmos}, a contradiction.

It remains to consider the case when $p/q = p/1 \in \Z$ and $S^3_{K}(p/q)$ is not reducible, and therefore Seifert fibred \cite{moser}.
For $S^3_{K'}(p/q)$ to be Seifert fibred, $Y(p/q)$ must be a solid torus.  This occurs only if $|p-rs|=1$, which implies that
$|p| = rs \pm 1$.

Therefore, there are at most $3$
shared surgeries between $K$ and $K'$.

\subsection{Composite knot}\label{tp} Let $K'$ be a composing knot.
By \cite[Theorem 2]{sorya}, there is an homeomorphism $S^3_K(p/q) \cong S^3_{K'}(p/q)$ only if $p/q = p/1 \in \Z$. By Lemma \ref{lemma:compositegluing}, $S^3_{K'}(p/1)$ is the gluing of two knot exteriors $S^3_{K_1}$ and $S^3_{K_2}$ along their boundaries. Therefore, $S^3_{K'}(p/1)$ contains an incompressible torus, namely $\del S^3_{K_1} = \del S^3_{K_2}$. Since $S^3_K(p/1)$ is atoroidal, $K$ and $K'$ share no common surgeries.
\section{Hyperbolic type}

Let $K$ be a knot of hyperbolic type and let $K' \neq K$ be another knot.

We define constants $\mathbf{q}(X)$ and $\mathbf{s}(X)$ where $X$ is a JSJ piece of a knot exterior. If $X$ is Seifert fibred, set $\mathbf{q}(X), \mathbf{s}(X) = 0$. If $X$ is hyperbolic, set
\[ 	\mathbf{q}(X) = 
\max\left\{{10.69},\sqrt{\frac{1.9793 \cdot 2 \pi}{\sys(X)} + 28.78} \right\} , \quad
\mathbf{s}(X) =
\max\left\{10.1, \sqrt{\frac{2\pi}{\sys(X)}+28.78}\right\}. \]

Let $Y,Y'$ respectively be the JSJ pieces of $S^3_K, S^3_{K'}$ containing $\partial S^3_K, \partial S^3_{K'}$. For  $X\in\{Y, Y'\}$, let $\hat{l}(\mu_X), \hat{l}(\lambda_X)$ be the normalized lengths of the meridian and longitude along the outermost boundary component of $X$ if $X$ is hyperbolic, and set $\hat{l}(\mu_X), \hat{l}(\lambda_X) = 0$ if $X$ is Seifert fibred.

Define
 \[\mathbf{s}(K, K')=\max \{\mathbf{s}(X') \ | \ X' \text{ a JSJ piece of } S^3_{K'} ,\; X' \not\supset\del S^3_{K'} ,\; |\del X'| = |\del Y|-1\}\]
and
\begin{alignat*}{2}
	\mathcal{Q}_\mu(K,K') &= \hat{l}(\mu_Y)\mathbf{q}(Y'), \quad \mathcal{S}_\mu(K,K') &= \hat{l}(\mu_Y)\mathbf{s}(K, K'),\\
	\mathcal{Q}_\lambda(K,K') &= \hat{l}(\lambda_Y)\mathbf{q}(Y'), \quad \mathcal{S}_\lambda(K,K') &= \hat{l}(\lambda_Y)\mathbf{s}(K, K'),
\end{alignat*}
and
\begin{align*}
	Q_\mu(K,K') &= \max\Big\{e(Y), e(Y'), \min\{\mathcal{Q}_\mu(K,K'), \mathcal{Q}_\mu(K',K)\} \Big\},\\
	S_\mu(K,K') &= \max\Big\{e(Y), e(Y'), \min\{	\mathcal{S}_\mu(K,K'),	\mathcal{S}_\mu(K',K)\Big\}, \\
	Q_\lambda(K, K') &= \max \Big\{a(Y), a(Y'), \: \min\{ \mathcal{Q}_\lambda(K,K') , \mathcal{Q}_\lambda(K',K) \} \Big\}, \\
	S_\lambda(K, K') &= \max \Big\{a(Y), a(Y'), \: \min\{ \mathcal{S}_\lambda(K,K') , \mathcal{S}_\lambda(K',K) \} \Big\}.
\end{align*}

These constants are closely related to $Q(K)$ and $S(K)$ from \cite[Thoerem 1.6(iv)]{sw}. In Theorem \ref{thm:main}, both knots $K$ and $K'$ are fixed, which allows us to use precise quantities instead of universal bounds.
\begin{lemma}\label{lemma:newQS}
 Let $Y$ be a hyperbolic JSJ piece of the exterior of a knot $K$. 
 Consider a slope $\sigma = p/q$ on the outermost boundary component of $Y$. 
 Let $X$ be a JSJ piece of a knot exterior. If either of the following holds
 \begin{enumerate}[(i)]
 	\item $|p| > \hat{l}(\lambda_Y)\mathbf{s}(X)$ and $|p| > a(Y)$;
  	\item $|q| > \hat{l}(\mu_Y)\mathbf{s}(X)$ and $|q|>e(Y)$,

 \end{enumerate}
 then $Y(p/q) \not\cong X$. 
\end{lemma}
\begin{proof}
	By Lemma \ref{lemma:hyperbolicfilling}, $Y(p/q)$ is hyperbolic, so assume $X$ is also hyperbolic.	Lemma \ref{lemma:length} combined with (i) or (ii) yields $\hat{l}(\sigma) \geq  10.1 \geq 7.823$ and $\hat{l}(\sigma') \geq \sqrt{\frac{2\pi}{\sys(X)}+28.78}$ from which the conclusion follows as in \cite[Lemma 3.1]{sw}.
\end{proof}

\subsection{Hyperbolic type}\label{hh} Suppose $K'$ is of hyperbolic type with outermost piece $Y'$ and suppose there is an orientation-preserving homeomorphism
$f: S^3_K(p/q) \rightarrow S^3_{K'}(p/q)$.

Suppose that $f(Y(p/q)) \neq Y'(p/q)$. If $|p| > \max \left\{a(Y), a(Y')\right\}$ or $|q|>\max \left\{e(Y), e(Y')\right\}$, then both $Y(p/q)$ and $Y'(p/q)$ are hyperbolic by Lemma \ref{lemma:hyperbolicfilling}. Therefore, $f(Y(p/q)) = X'$ for some JSJ piece $X' \in S^3_{K'}\setminus Y'$ and $f\inv(Y'(p/q)) = X$ for some JSJ piece $X \in S^3_{K}\setminus Y$. Applying Lemma \ref{lemma:newQS} to both $Y$ and $Y'$ gives that $|p|\leq S_\lambda(K,K')$ and $|q|\leq S_\mu(K,K')$.

We now consider the possibility that $Y(p/q)$ and $Y'(p/q)$ are mapped to one another. The following lemma treats a slightly more generalized version.

\begin{lemma}\label{lemma:wakelinQ}
	Let $K, K'$ be distinct knots of hyperbolic type with outermost JSJ pieces $Y, Y'$ respectively. Suppose there is an orientation-preserving homeomorphism $f: S^3_K(p/q) \rightarrow S^3_{K'}(p/q')$ such that $f(Y(p/q)) = Y'(p/q')$. Then $|p|\leq Q_\lambda(K,K')$, and $|q| \leq Q_\mu(K,K')$ or $|q'| \leq Q_\mu(K,K')$.
\end{lemma}
\begin{proof}
The proof if based on that of \cite[Proposition 4.8]{wakelin}.

If $|q|,|q'|> Q_\mu(K,K')$ or $|p|> Q_\lambda(K,K')$, then $|q|,|q'|>\max \{e(Y), e(Y')\}$ or $|p|> \max \{a(Y), a(Y')\}$. So in both cases, $Y(p/q)$ and $Y'(p/q)$ are hyperbolic according to Lemma \ref{lemma:hyperbolicfilling}.

Suppose $|q|,|q'| > Q_\mu(K,K')$. Consider the slope $\sigma = p/q$ along the outermost boundary component of $Y$. 

As in the proof of Lemma \ref{lemma:newQS}, we have $\hat{l}(\sigma) \geq  10.1 \geq 7.823$. Therefore, by \cite[Theorem 5.17, Corollary 6.13]{fps}, the core of the surgery solid torus $v \subset Y(p/q)$ is a geodesic satisfying
\[l(v) \leq \frac{2\pi}{\hat{l}(\sigma)^2-28.78}
\leq \frac{2\pi}{\left(\frac{|q|}{\hat{l}(\mu_Y)}\right)^2-28.78}
\leq \frac{2\pi}{\left(\frac{\hat{l}(\mu_Y)\mathbf{q}(Y)}{\hat{l}(\mu_Y)}\right)^2-28.78}
\leq \frac{2\pi}{\mathbf{q}(Y)^2-28.78}, \]
which is less than or equal to $0.0735$ and $1.9793\inv \sys(Y)$. The same argument implies that the core of the surgery solid torus $v' \subset Y'(p/q')$ also has length less than or equal to $0.0735$ and $1.9793\inv \sys(Y)$. As in the proof of \cite[Proposition 4.8]{wakelin}, this implies that $f(v)=v'$ and therefore $K=K'$, a contradiction.

Suppose now that $|p|> Q_\lambda(K,K')$. The chain of inequalities above becomes 
\[l(v) \leq \frac{2\pi}{\hat{l}(\sigma)^2-28.78}
\leq \frac{2\pi}{\left(\frac{p}{\hat{l}(\lambda_Y)}\right)^2-28.78}
\leq \frac{2\pi}{\left(\frac{\hat{l}(\lambda_Y)\mathbf{q}(Y)}{\hat{l}(\lambda_Y)}\right)^2-28.78}
\leq \frac{2\pi}{\mathbf{q}(Y)^2-28.78} \]
which leads to the same contradiction.
\end{proof}

Therefore, $K$ and $K'$ have orientation-preserving homeomorphic $p/q$-surgeries for at most 
\[n\left(\max\{Q_\mu(K,K'), S_\mu(K,K')\}, \max\{Q_\lambda(K,K'), S_\lambda(K,K')\}\right)\]
slopes $p/q$.

\subsection{Cable knot}\label{hc} Suppose $K'$ is an $(r',s')$-cable knot $C_{r',s'}(J')$, where $|s'|\geq 2$ is the winding number of the cable. Suppose there is an orientation-preserving homeomorphism $f: S^3_K(p/q) \rightarrow S^3_{K'}(p/q)$.

Let $Y'$ be the outermost $(r',s')$-cable space of $S^3_{K'}$. By Lemma \ref{lemma:newQS}, if $|p| > a(Y)$ or $|q| > e(Y)$, then $Y(p/q)$ is mapped to a JSJ piece of $S^3_{K'}\setminus Y'$ only if $|q|\leq \mathcal{S}_\mu(K, K')$ or $|p|\leq \mathcal{S}_\lambda(K, K')$.

So suppose $|q| > \max\{e(Y), \mathcal{S}_\mu(K, K')\}$ and $|p| > \max\{a(Y), \mathcal{S}_\lambda(K, K')\}$. By Lemmas \ref{lemma:hyperbolicfilling} and \ref{lemma:newQS}, $Y(p/q)$ is mapped to the JSJ piece of $S^3_{K'}(p/q)$ that contains the surgery solid torus, which must thus be hyperbolic.

Then $Y'(p/q)$ must be a solid torus, otherwise it would be a Seifert fibred JSJ piece containing the surgery solid torus, a contradiction. Furthermore, $J'$ is of hyperbolic type, with outermost hyperbolic piece $X$ such that $Y(p/q) \cong X(p/qs'^2)$.

By Lemma \ref{lemma:wakelinQ}, we have 
$|p|\leq Q_\lambda(K,J')$, and $|q| \leq Q_\mu(K,J')$ or $|qs'^2| \leq Q_\mu(K,J')$. But since $s'^2 \geq 4$, the set of slopes $p/q$ with $|qs'^2| \leq Q_\mu(K,J')$ is a subset of the set of slopes with $|q| \leq Q_\mu(K,J')$.

In conclusion, $K$ and $K' = C_{r,s}(J)$ share at most
\[n \Big(\max\{a(Y), \mathcal{S}_\lambda(K,K'), Q_\lambda(K,J')\}, \: \max\{e(Y), \mathcal{S}_\mu(K,K'), Q_\mu(K,J')\} \Big)\]

\subsection{Composite knot}\label{hp} Let $K'$ be a composing space.
By \cite[Theorem 2]{sorya}, $K$ and $K'$ have homeomorphic $p/q$-surgeries only if $p/q$ is an integer $p/1$. 
By Lemma \ref{lemma:compositegluing}, if $S^3_K(p/1) \cong S^3_{K'}(p/1)$, then $Y(p/1)$ is homeomorphic to a JSJ piece $X$ of $S^3_{K'}$.

If $X$ is Seifert fibred, then $|p| \leq a(Y)$.
If $X$ is hyperbolic, then $|p| \leq \mathcal{S}_\lambda(K,K')$ by Lemma \ref{lemma:newQS}.
Therefore, the number of shared Dehn surgeries between $K$ and $K'$ is at most 
\[2 \cdot \lfloor \max \left\{a(Y), \mathcal{S}_\lambda(K,K')\right\}\rfloor.\]

\section{Cable knot}

Let $K$ be an $(r,s)$-cable knot $C_{r,s}(J)$, where $|s| \geq 2$ is the winding number of the cable.

We give a more precise version of \cite[Proposition 3.6]{sorya} for cable knots. It says that the surgery solid torus of a Dehn surgery is contained in a single JSJ piece -- the {\em surgered piece} --, except possibly in certain specific cases.

\begin{prop}\label{prop:jsjsurgered}
	Let $K = C_{r,s}(J)$ be a cable knot whose exterior has JSJ decomposition $Y_0 \cup Y_1 \cup \ldots Y_k$, where $Y_0$ is the outermost cable space of $S^3_K$ and $Y_1$ the outermost JSJ piece of $S^3_J$. Then for all but at most 4 slopes $p/q$, $S^3_{K}(p/q)$ is irreducible and its JSJ decomposition is either
	\[Y_0(p/q) \cup Y_1 \cup \ldots \cup Y_k \quad \text{or} \quad Y_1(p/qs^2) \cup Y_2 \cup \ldots \cup Y_k,\]
	the second scenario occurring precisely when $|p-qrs|=1$.
	
	The 4 slopes for which this may not occur are $rs/1,(rs+1)/1, (rs-1)/1$ and $4ab/1$ if $K$ is the cable of an $(a,b)$-torus knot or $(a,b)$-cable knot.
\end{prop}
\begin{proof}
	This is a consequence of the proof of  \cite[Proposition 3.6]{sorya}, which describes the three situations in which the conclusion may fail to occur. The first case is when $p/q = rs/1$; then $S^3_K(p/q)$ is reducible. The second case is when $Y_0(p/q)$ is a twisted $I$-bundle over the Klein bottle: its Seifert fibred structure may agree with that of $Y_1$ on $Y_0(p/q) \cap Y_1$ only if $Y_1$ is an $(a,b)$-torus knot or $(a,b)$-cable space and $p/q = 4ab/1$. The last case is when both $Y_0(p/q)$ and $Y_1(p/qs^2)$ are solid tori: we then have $q=1$ and $p=rs\pm1$.
\end{proof}

\subsection{Cable knot}\label{cc} Let $K'$ be an $(r',s')$-cable knot $C_{r',s'}(J')$.
Suppose there is an orientation-preserving homeomorphism $f: S^3_{K}(p/q) \rightarrow S^3_{K'}(p/q)$ and let $Y, Y'$ respectively be the outermost cable spaces pieces of $S^3_K, S^3_{K'}$.

We first consider the case $|p-qrs|, |p-qr's'|>1$. Suppose $p/q$ is not one of the slopes yielding a reducible surgery or a twisted $I$-bundle over the Klein bottle from Proposition \ref{prop:jsjsurgered}.

If $f(Y(p/q))=Y'(p/q)$, then $K = K'$ by \cite[Proposition 6.7]{sorya}, a contradiction. Thus, $f(Y(p/q))=X$ for $X$ a JSJ piece of $S^3_{K'} \setminus Y'$.

Let
\[T_s(K')=\big\{ T_{a, b} \text{ torus knot} \:\big|\: S^3_{T_{a,b{}}} \subset S^3_{K'}, |s| \in \{|a|,|b|\} \big\}.\]

By \cite[Proposition 3.5]{sorya}, $X$ is the exterior of a torus knot contributing to $T_s(K')$, say $T_{a,s}$. Therefore, $|p-qrs|=|a|$. Similarly, $Y'(p/q)$ is homeomorphic to the exterior of a torus knot $T_{a', s'} \in T_{s'}(K)$ and  $|p-qr's'|=|a'|$. Therefore $q = (\epsilon a + \epsilon' a')/(rs-r's')$ and $p = (\epsilon a r's' - \epsilon' a'rs)/(rs-r's')$ where $\epsilon, \epsilon' = \pm 1$. This gives 4 possible slopes for each pair $(a, a')$, bounding from above the number of shared surgeries between $K$ and $K'$ by 
\[4 \cdot \#T_{s'}(K) \cdot \#T_s(K') + 2.\]


We now consider the case where $|p-qrs|=1$, without loss of generality. Let $X, X'$ respectively be the outermost JSJ pieces of $S^3_J, S^3_{J'}$. Note $|p-qrs|=1$ implies that we obtain the JSJ decomposition of Proposition \ref{prop:jsjsurgered} for $S^3_K(p/q)$ and $S^3_{K'}(p/q)$ except possibly when $X$ or $X'$ are cable spaces.

We examine each possibility for $X$ in detail. 

\subsubsection{Torus knot}
Suppose $X$ is the exterior of a torus knot $T_{a,b}$. Then by the arguments in \cite[Section 6.1]{sorya}, we have that $K'$ is the cable of a torus knot $T_{a',b'}$ and $|p-qr's'|=1$. We compute $q = (\epsilon + \epsilon c')/(s^2ab-rs), p = (\epsilon s^2ab - \epsilon' c'rs)/(s^2ab-rs)$ for $\epsilon, \epsilon' = \pm 1$ and $c' \in \{a', b'\}$. This gives at most 8 possibilities of slopes $p/q$ for this case.

\subsubsection{Hyperbolic}
Suppose $X$ is hyperbolic. By Lemma \ref{lemma:newQS}, if $f(X(p/qs^2)$ is not the surgered piece of $S^3_{K'}(p/q)$, then $|p| \leq S_\lambda(J,K')$ and $|q| \leq S_\mu(J,K')/s^2$.

If $f(X(p/qs^2)$ is the surgered piece of $S^3_{K'}(p/q)$, then by \cite[Corollary 3.9]{sorya} it corresponds to $X'(p/qs'^2)$ where $X'$ is hyperbolic and $|p-qr's'|=1$. By Lemma \ref{lemma:hyperbolicfilling}, $|p| \leq Q_\lambda(J,J')$, and $|q| \leq \max \{Q_\mu(J, J')/s^2, Q_\mu(J, J')/s'^2\} = Q_\mu(J, J')/\min\{s, s'\}^2$, giving
\[n \Big(\max\{S_\lambda(J,K'),Q_\lambda(J,J')\},  \max\left\{\frac{S_\mu(J,K')}{s^2}, \frac{Q_\mu(J,J')}{\min \{s, s'\}^2} \right\} \Big)\]
possible slopes for this case.

\subsubsection{Cable space}
Suppose $X$ is an $(u,v)$-cable space. Suppose $p/q \neq (rs \pm 1)/1$ so that $S^3_{K}(p/q)$ admist a JSJ decomposition as in Proposition \ref{prop:jsjsurgered}. Then $X(p/qs^2)$ is Seifert fibred over the disc with two cone points of orders $|v|$ and $|p-qs^2uv|$. 
If $f(X(p/qs^2))$ is the surgered piece of $S^3_{K'}(p/q)$, then the arguments from the proof of \cite[Proposition 6.7]{sorya} imply that $K=K'$, a contradiction\footnote{\cite[Proposition 6.7]{sorya} requires $|q|>2$ to invoke the more general \cite[Corollary 3.9]{sorya}, but in our specific situation, its proof holds for all $p/q$ that avoid the special slopes of Proposition \ref{prop:jsjsurgered}.}.

If $f(X(p/qs^2))$ is not the surgered piece of $S^3_{K'}(p/q)$, then it is homeomorphic to the exterior of a knot $T_{a', b'} \in T_{v}(J')$ such that $\{|a'|,|b'|\} = \{|v|, |p-qs^2uv|\}$. Combining this with $|p-qrs|=1$, we obtain $4 \cdot  \# T_{v}(J')$
possible slopes for which this situation may occur.

\subsubsection{Composing space}
Suppose $X$ is a composing space. By \cite[Lemma 5.4]{sorya}, since $|qs^2|>1$, $f(X(p/qs^2))$ is the surgered piece of $S^3_{K'}(p/q)$. Then by \cite[Corollary 3.9]{sorya}, the surgered piece of $S^3_{K'}(p/q)$ is $X'(p/qs'^2)$ where $X' \cong X, |s'| = |s|$ and $|p-qr's'| = 1$. By \cite[Lemma 6.2]{sorya}, $r/s = r'/s'$ and therefore $K=K'$, a contradiction.

 \subsection{Composite knot}\label{cp} Suppose $K'$ is a composite knot. By \cite[Theorem 2]{sorya}, $p/q$ is an integer $p/1$. By Lemma \ref{lemma:compositegluing}, the JSJ pieces of $S^3_{K'}(p/1)$ are those of $S^3_{K'} \setminus Y'$. Therefore, if $Y(p/1)$ is not a solid torus, which occurs for the 2 slopes such that $|p-rs|=1$, $f(Y(p/1))$ is the exterior of a torus knot $T_{a', b'}$ in $S^3_{K'}$ such that $\{|a'|, |b'|\} = \{|s|, |p-rs|\}$. There are at most $2 \cdot \# T_s(K')$ slopes for which this can happen.

\section{Composite knot}\label{pp}

Let $K = K_1 \# K_2$, where at least one of $K_1, K_2$ is prime.

Let $K' = J_1' \# J_2' \# \ldots \# J_n'$ where the $J_i'$ are prime for all $i=1,\ldots, n$ and $K' \neq K$. By Lemma \ref{lemma:compositegluing}, $S^3_{K'}(p/1)$ can be described as the gluing of the exterior of $J_i'$ for any $i = 1, \ldots, n$ to the exterior of $J_{i_1}' \# \ldots \# J_{i_{n-1}}'$, where $\{i_1, \ldots, i_{n-1}\} = \{1, \ldots, n\} \setminus \{i\}$, along a JSJ torus $T_i'=\del S^3_{J_i'}$. 

Suppose there is an orientation preserving homeomorphism
$f: S^3_K(p/1) \rightarrow S^3_{K'}(p/1)$. If $f(T) = T_i'$ for some $i = 1, \ldots, n$, then $J_i' = K_1$ and $J_{i_1}' \# \ldots \# J_{i_{n-1}}' = K_2$, from which we conclude that $K = K'$. Our goal is thus to determine for how many slopes $p/1$ we may have $f(T) \neq T_i'$ for all $i = 1, \ldots, n$.

Suppose that this is the case. Then, without loss of generality, $f(T)$ is the boundary of a submanifold $S^3_{K_1}\subset S^3_{K'}$ and $K'$ has a summand $K_1'$ which is a satellite of $K_1$. Simultaneously, $K_2$ is a satellite of a summand $K_2'$ of $K' = K_1' \# K_2'$ where $K_1'$ or $K_2'$ is prime.

We consider each possibility for the outermost piece $Y$ of $S^3_{K_2}$, which is not the exterior of a torus knot since $K_2$ is a satellite knot.
Let $\mu_i$ and $\lambda_i$ be the meridian and longitude of $K_i, i=1,2$.

\subsection{Cable space}
Suppose $Y$ is an $(r,s)$-cable space, where $|s|\geq 2$ is the winding number of the cable.
We have that $f(Y)$ is homeomorphic to an $(r',s)$-cable space $V \subset S^3_{K_1'}$.
In the coordinates $\mu_1', \lambda_1'$ given by $f(S^3_{K_1})$,a regular fibre of $V$ along $f(T)$ has slope $r'\mu_1' + s\lambda_1'$. By the knot complement theorem, we have $f(\mu_1) = \mu_1'$ and $f(\lambda_1)=\lambda_1'$. Applying Lemma \ref{lemma:compositegluing}, this implies that $f ((rs-p)\mu_1 - \lambda_1) = (rs-p)\mu_1' - \lambda_1' = r'\mu_1' + s\lambda_1'$. Therefore, $s = -1$, a contradiction.

\subsection{Hyperbolic type}

Suppose $Y$ is hyperbolic. By Lemma \ref{lemma:compositegluing}, we have $\mu_1 = \mu_2$ and $\lambda_1 = -p\mu_2 - \lambda_2$ along $T$. Therefore, the slope $\lambda_1$ has length $l(-p\mu_2 - \lambda_2)$ along $T$ seen as a boundary component of the hyperbolic manifold $Y$. By Mostow-Prasad rigidity, we have that $f(\lambda_1)$ is a slope of length $l(-p\mu_2 - \lambda_2)$ along $f(T)$ seen as a boundary component of the hyperbolic manifold $f(Y)$. 

Let $M \cong S^3_{K_1} \cup Y'$ be a submanifold of $S^3_{K'}$, where $Y'\cong Y$ as manifolds, but their corresponding links in the satellite construction of $K'$ and $K$ may differ. The longitude of $K_1$ is identified to a slope of a specific length $L$ along a boundary component of $Y$. If $f(T) = S^3_{K_1} \cap Y' \subset M$, then $L = l(-p\mu_2 - \lambda_2)$ by the knot complement theorem and the above. Since $L$ and the geometry of $Y$ are fixed, this equality holds for at most two values of $p$, namely the integral roots of the degree 2 polynomial 
\[l(\mu_2)^2 p^2 -2l(\mu_2)l(\lambda_2)(\cos \theta) p+l(\lambda_2)^2-L^2,\]
where $\theta$ is the angle between $\lambda_2$ and $\mu_2$.

\subsection{Composing space}
Suppose $Y$ is a composing space. Then $K_2$ is a composite knot with a (possibly trivial) satellite $P(K_2')$ of $K_2'$ as a prime summand. The exterior of $K_1'$ contains the composing space $f(Y)$. 
If $f(Y)$ is the outermost piece of $S^3_{K_1'}$, then $K_1'$ is a composite knot whose prime summands consist of $K_1$ and the summands of $K_2$ except $P(K_2')$. Therefore, $K = K'$, a contradiction.

Hence, $K_1'$ is a satellite of a composite knot whose exterior has outermost piece $f(Y)$, and we may assume that $K_1'$ is prime. We thus find ourselves in the situation of the previous cases, with $K_1, K_2$ replaced by $K_2', K_1'$.

All cases considered, we obtain that the prime summand of $K'$ that contains $f(Y)$ must be of hyperbolic type. Therefore, an orientation-preserving homeomorphism $f: S^3_K(p/1) \rightarrow S^3_{K'}(p/1)$ may occur only when $K, K'$ respectively have hyperbolic type summands $J, J'$ with outermost pieces $Y, Y'$ such that $f(Y) \subset S^3_{J'}$ and $f\inv(Y') \subset S^3_{J}$.

In conclusion, for $K=J_1 \# J_2 \# \ldots \# J_m$ and $K=J_1' \# J_2' \# \ldots \# J_n'$, we may take $D(K,K')$ to be
\[2 \cdot \max_{\substack{i=1,\ldots, m \\ k=1,\ldots,n}} \left\{\# \{J_i \text{ hyperbolic type}\}, \# \{J_k' \text{ hyperbolic type} \} \right\}\]

This concludes the proof of Theorem \ref{thm:main}. \qed
\clearpage
\newgeometry{bottom=0.5cm}

\section{Summary}
\vspace{3em}
\renewcommand{\arraystretch}{2}

\begin{table}[h!]
	\rotatebox{90}{
		\begin{minipage}{8in}
        \vspace{-3em}
	\begin{tabular}{l||l|l|l|l}
		$K$ \textbackslash{} $K'$ & Torus & Hyperbolic type & Cable $C_{r',s'}(J')$ & Composite $J_1' \# \ldots \# J_n'$\\ \hline\hline
		Torus                     &    2   & $n(a(Y'), e(Y'))$             &       3    &     0      \\ \hline
		\renewcommand{\arraystretch}{1}\begin{tabular}{@{}l@{}}Hyperbolic\\ type\end{tabular}\renewcommand{\arraystretch}{2}                &       & \renewcommand{\arraystretch}{1}\begin{tabular}{@{}l@{}} $n\big(\max\{Q_\lambda(K,K'), S_\lambda(K,K')\}$,\\ $\max\{Q_\mu(K,K'), S_\mu(K,K')\}\big)$\end{tabular}\renewcommand{\arraystretch}{2}  &  \renewcommand{\arraystretch}{1}\begin{tabular}{@{}l@{}} $n \big(\max\{a(Y), \mathcal{S}_\lambda(K,K'),Q_\lambda(K,J')\}, $ \\ $ \max\{e(Y), \mathcal{S}_\mu(K,K'), Q_\mu(K,J')\} \big)$\end{tabular}\renewcommand{\arraystretch}{2}    &     $2 \cdot \lfloor \max \left\{a(Y), \mathcal{S}_\lambda(K,K')\right\}\rfloor$     \\ \hline
		\renewcommand{\arraystretch}{1}\begin{tabular}{@{}l@{}}Cable\\ $C_{r,s}(J)$\end{tabular}\renewcommand{\arraystretch}{2}             &       &            &   \begin{tabular}{@{}l@{}|l}
			$J$	& $4 \cdot \#T_{s'}(K) \cdot \#T_s(K') + 2 +$\\ \hline\hline
			Torus & 8 \\ \hline
			\renewcommand{\arraystretch}{1}\begin{tabular}{@{}l@{}}Hyperbolic\\ type\end{tabular}\renewcommand{\arraystretch}{2}  & \renewcommand{\arraystretch}{1}\begin{tabular}{@{}l@{}} $n \Big(\max\{S_\lambda(J,K'),Q_\lambda(J,J')\}, $ \\ $ \max\left\{\frac{S_\mu(J,K')}{s^2}, \frac{Q_\mu(J,J')}{\min \{s, s'\}^2} \right\} \Big)$\end{tabular}\renewcommand{\arraystretch}{2} \\ \hline
			Cable $C_{u,v}(\hat{J})$& $4 \cdot  \# T_{v}(J') + 2$ \\ \hline
			Composite & 0
		\end{tabular}
		    &    $2\cdot T_s(K') + 2$       \\ \hline
		\renewcommand{\arraystretch}{1}\begin{tabular}{@{}l@{}}Composite\\ $J_1 \# \ldots \# J_m$\end{tabular}\renewcommand{\arraystretch}{2}                 &       &            &       &   \renewcommand{\arraystretch}{1}\begin{tabular}{@{}l@{}}$2 \cdot \max \big\{\# \{J_i \text{ hyperbolic type}\},$ \\ $\# \{J_k' \text{ hyperbolic type} \} \big\}$\end{tabular}\renewcommand{\arraystretch}{2}     
	\end{tabular}
	\caption{$D(K,K')$ for each combination}\label{table:D}

	\vspace{2em}
	\begin{minipage}[t]{0.5\textwidth}
    	$n(B,C) = \# \big\{ (p,q) \in \Z_{\geq 0} \times \Z\setminus\{0\} \;\big|\; \gcd(p,q)=1,  p\leq B,  |q| \leq C \big\}$\\
            \\
        $Y$ (resp. $Y'$) is the outermost JSJ piece of $S^3_K$ (resp. $S^3_{K'}$)\\
        $a(Y) = 6 \cdot l(\lambda_Y)/A(\del S^3_K)$ \\
        $e(Y) = \begin{cases}
        \min\{8, \frac{6l(\mu_Y)}{A(T)}\} \;, &|\del Y| = 1;\\
        \min\{2, \frac{6l(\mu_Y)}{A(T)}\} \;, &|\del Y| > 1.
    \end{cases}$ \\
        \\
		$T_s(K')=\big\{ T_{a, b} \text{ torus knot} \:\big|\: S^3_{T_{a,b{}}} \subset S^3_{K'}, |s| \in \{|a|,|b|\} \big\}$\\
		\\
		$\mathbf{s}(K, K')=\max \left\{\mathbf{s}(X') \enspace\left| \enspace
        \begin{aligned}
            &X' \text{ a JSJ piece of } S^3_{K'} \\
            &X' \not\supset\del S^3_{K'} \\
            &|\del X'| = |\del Y|-1 
        \end{aligned}\right.\right\}$
	\end{minipage} 
	\hfill
	\begin{minipage}[t]{0.48\textwidth}
    $\mathbf{q}(X) = 
		\max\left\{{10.69},\sqrt{\frac{1.9793 \cdot 2 \pi}{\sys(X)} + 28.78} \right\}\\
		\mathbf{s}(X) =
		\max\left\{10.1, \sqrt{\frac{2\pi}{\sys(X)}+28.78}\right\}$\\
            \\
		$\mathcal{Q}_\mu(K,K') = \hat{l}(\mu_Y)\mathbf{q}(Y')$, \;
		$\mathcal{S}_\mu(K,K') = \hat{l}(\mu_Y)\mathbf{s}(K,K')$\\
		$\mathcal{Q}_\lambda(K,K') = \hat{l}(\lambda_Y)\mathbf{q}(Y')$, \; $\mathcal{S}_\lambda(K,K') = \hat{l}(\lambda_Y)\mathbf{s}(K,K')$\\
            \\
        $Q_\mu(K,K') = \max\Big\{e(Y), e(Y'), \min\{\mathcal{Q}_\mu(K,K'), \mathcal{Q}_\mu(K',K)\} \Big\}$\\
		$S_\mu(K,K') = \max\Big\{e(Y), e(Y'), \min\{	\mathcal{S}_\mu(K,K'),	\mathcal{S}_\mu(K',K)\Big\}$ \\
		$Q_\lambda(K, K') = \max \Big\{a(Y), a(Y'), \: \min\{ \mathcal{Q}_\lambda(K,K') , \mathcal{Q}_\lambda(K',K) \} \Big\}$ \\
		$S_\lambda(K, K') = \max \Big\{a(Y), a(Y'), \: \min\{ \mathcal{S}_\lambda(K,K') , \mathcal{S}_\lambda(K',K) \} \Big\}$
	\end{minipage}

\end{minipage}
}

\end{table}	
\restoregeometry

\clearpage



\bibliographystyle{amsalpha} 
\bibliography{bib}

@misc{bkm,
      title={The search for alternating surgeries}, 
      author={Kenneth L. Baker and Marc Kegel and Duncan McCoy},
      year={2026},
      eprint={2409.09842},
      archivePrefix={arXiv},
      primaryClass={math.GT},
      url={https://arxiv.org/abs/2409.09842}, 
}

@misc{kp,
      title={Knots that share four surgeries}, 
      author={Marc Kegel and Lisa Piccirillo},
      year={2025},
      eprint={2505.13168},
      archivePrefix={arXiv},
      primaryClass={math.GT},
      url={https://arxiv.org/abs/2505.13168}, 
}

@article {sw,
    AUTHOR = {Sorya, Patricia and Wakelin, Laura},
     TITLE = {Effective bounds on characterising slopes for all knots},
   JOURNAL = {Trans. Amer. Math. Soc.},
      YEAR = {to appear, 2026},
       DOI = {10.1090/tran/9673},
       URL = {https://doi.org/10.1090/tran/9673},
}

@article {sorya,
    AUTHOR = {Sorya, Patricia},
     TITLE = {Characterizing slopes for satellite knots},
   JOURNAL = {Adv. Math.},
  FJOURNAL = {Advances in Mathematics},
    VOLUME = {450},
      YEAR = {2024},
     PAGES = {109746},
      ISSN = {0001-8708,1090-2082},
   MRCLASS = {57K10},
  MRNUMBER = {4755445},
       DOI = {10.1016/j.aim.2024.109746},
       URL = {https://doi.org/10.1016/j.aim.2024.109746},
}

@article {wakelin,
	AUTHOR = {Wakelin, Laura},
	TITLE = {Characterising slopes for hyperbolic knots and {W}hitehead
	doubles},
	JOURNAL = {Algebr. Geom. Topol.},
	FJOURNAL = {Algebraic \& Geometric Topology},
	VOLUME = {26},
	YEAR = {2026},
	NUMBER = {2},
	PAGES = {625--657},
	ISSN = {1472-2747,1472-2739},
	MRCLASS = {57K10 (57K32)},
	MRNUMBER = {5034310},
	DOI = {10.2140/agt.2026.26.625},
	URL = {https://doi.org/10.2140/agt.2026.26.625},
}

@article{agol,
	AUTHOR = {Agol, Ian},
	TITLE = {Bounds on exceptional {D}ehn filling},
	JOURNAL = {Geom. Topol.},
	FJOURNAL = {Geometry and Topology},
	VOLUME = {4},
	YEAR = {2000},
	PAGES = {431--449},
	ISSN = {1465-3060,1364-0380},
	MRCLASS = {57M50 (57M25 57M27 57S25)},
	MRNUMBER = {1799796},
	MRREVIEWER = {Danny\ C.\ Calegari},
	DOI = {10.2140/gt.2000.4.431},
	URL = {https://doi.org/10.2140/gt.2000.4.431},
}

@article {lackenby_6,
	AUTHOR = {Lackenby, Marc},
	TITLE = {Word hyperbolic {D}ehn surgery},
	JOURNAL = {Invent. Math.},
	FJOURNAL = {Inventiones Mathematicae},
	VOLUME = {140},
	YEAR = {2000},
	NUMBER = {2},
	PAGES = {243--282},
	ISSN = {0020-9910,1432-1297},
	MRCLASS = {57M07 (20F65 20F67 57M05 57N10)},
	MRNUMBER = {1756996},
	MRREVIEWER = {William\ H.\ Jaco},
	DOI = {10.1007/s002220000047},
	URL = {https://doi.org/10.1007/s002220000047},
}

@article {lackenby,
	AUTHOR = {Lackenby, Marc},
	TITLE = {Every knot has characterising slopes},
	JOURNAL = {Math. Ann.},
	FJOURNAL = {Mathematische Annalen},
	VOLUME = {374},
	YEAR = {2019},
	NUMBER = {1-2},
	PAGES = {429--446},
	ISSN = {0025-5831,1432-1807},
	MRCLASS = {57M25},
	MRNUMBER = {3961316},
	MRREVIEWER = {Brandy\ Guntel\ Doleshal},
	DOI = {10.1007/s00208-018-1757-x},
	URL = {https://doi.org/10.1007/s00208-018-1757-x},
}

@article {mccoy,
	AUTHOR = {McCoy, Duncan},
	TITLE = {Non-integer characterizing slopes for torus knots},
	JOURNAL = {Comm. Anal. Geom.},
	FJOURNAL = {Communications in Analysis and Geometry},
	VOLUME = {28},
	YEAR = {2020},
	NUMBER = {7},
	PAGES = {1647--1682},
	ISSN = {1019-8385,1944-9992},
	MRCLASS = {57K10},
	MRNUMBER = {4184829},
	MRREVIEWER = {Ana\ G.\ Lecuona},
	DOI = {10.4310/CAG.2020.v28.n7.a5},
	URL = {https://doi.org/10.4310/CAG.2020.v28.n7.a5},
}

@article {budney,
	AUTHOR = {Budney, Ryan},
	TITLE = {J{SJ}-decompositions of knot and link complements in {$S^3$}},
	JOURNAL = {Enseign. Math. (2)},
	FJOURNAL = {L'Enseignement Math\'ematique. Revue Internationale. 2e
	S\'erie},
	VOLUME = {52},
	YEAR = {2006},
	NUMBER = {3-4},
	PAGES = {319--359},
	ISSN = {0013-8584},
	MRCLASS = {57M25 (57N10)},
	MRNUMBER = {2300613},
	MRREVIEWER = {Patrick\ Popescu-Pampu},
}

@article {fps,
	AUTHOR = {Futer, David and Purcell, Jessica S. and Schleimer, Saul},
	TITLE = {Effective bilipschitz bounds on drilling and filling},
	JOURNAL = {Geom. Topol.},
	FJOURNAL = {Geometry \& Topology},
	VOLUME = {26},
	YEAR = {2022},
	NUMBER = {3},
	PAGES = {1077--1188},
	ISSN = {1465-3060},
	MRCLASS = {57K10 (30F40 57K32)},
	MRNUMBER = {4466646},
	DOI = {10.2140/gt.2022.26.1077},
	URL = {https://doi.org/10.2140/gt.2022.26.1077},
}

@article {moser,
	AUTHOR = {Moser, Louise},
	TITLE = {Elementary surgery along a torus knot},
	JOURNAL = {Pacific J. Math.},
	FJOURNAL = {Pacific Journal of Mathematics},
	VOLUME = {38},
	YEAR = {1971},
	PAGES = {737--745},
	ISSN = {0030-8730,1945-5844},
	MRCLASS = {55F55 (57C45)},
	MRNUMBER = {383406},
	MRREVIEWER = {Wilbur\ Whitten},
	URL = {http://projecteuclid.org/euclid.pjm/1102969920},
}

@article {gordon_sat,
    AUTHOR = {Gordon, C. McA.},
     TITLE = {Dehn surgery and satellite knots},
   JOURNAL = {Trans. Amer. Math. Soc.},
  FJOURNAL = {Transactions of the American Mathematical Society},
    VOLUME = {275},
      YEAR = {1983},
    NUMBER = {2},
     PAGES = {687--708},
      ISSN = {0002-9947,1088-6850},
   MRCLASS = {57M25 (57N10)},
  MRNUMBER = {682725},
MRREVIEWER = {R.\ C.\ Kirby},
       DOI = {10.2307/1999046},
       URL = {https://doi.org/10.2307/1999046},
}

@article {kmos,
    AUTHOR = {Kronheimer, P. and Mrowka, T. and Ozsv\'ath, P. and Szab\'o,
              Z.},
     TITLE = {Monopoles and lens space surgeries},
   JOURNAL = {Ann. of Math. (2)},
  FJOURNAL = {Annals of Mathematics. Second Series},
    VOLUME = {165},
      YEAR = {2007},
    NUMBER = {2},
     PAGES = {457--546},
      ISSN = {0003-486X,1939-8980},
   MRCLASS = {57R58 (57M27 57R57)},
  MRNUMBER = {2299739},
MRREVIEWER = {Vicente\ Mu\~noz},
       DOI = {10.4007/annals.2007.165.457},
       URL = {https://doi.org/10.4007/annals.2007.165.457},
}

@article {lm,
    AUTHOR = {Lackenby, Marc and Meyerhoff, Robert},
     TITLE = {The maximal number of exceptional {D}ehn surgeries},
   JOURNAL = {Invent. Math.},
  FJOURNAL = {Inventiones Mathematicae},
    VOLUME = {191},
      YEAR = {2013},
    NUMBER = {2},
     PAGES = {341--382},
      ISSN = {0020-9910,1432-1297},
   MRCLASS = {57R65},
  MRNUMBER = {3010379},
MRREVIEWER = {Yu\ Zhang},
       DOI = {10.1007/s00222-012-0395-2},
       URL = {https://doi.org/10.1007/s00222-012-0395-2},
}

\end{document}